\documentclass[11pt]{article}
\usepackage{amsmath,amssymb}
\usepackage{epsfig}
\usepackage{amsfonts}
\usepackage{latexsym}
\usepackage{amsthm}
\usepackage{enumerate}
\usepackage{pgf}
\usepackage{tikz}

\usepackage{mathrsfs}

\usetikzlibrary{arrows,automata}

\newtheorem{theorem}{Theorem}[section]
\newtheorem{lemma}[theorem]{Lemma}
\newtheorem{proposition}[theorem]{Proposition}
\newtheorem{corollary}[theorem]{Corollary}

\theoremstyle{definition}
\newtheorem{definition}[theorem]{Definition}
\newtheorem{example}[theorem]{Example}

\theoremstyle{remark}
\newtheorem{remark}[theorem]{Remark}

\numberwithin{equation}{section}

\newcommand{\N}{\ensuremath{\mathbb{N}}}

\newcommand{\A}{\ensuremath{\mathcal{F}}}

\renewcommand{\c}{ {\mathbf{c}}}
\renewcommand{\d}{ {\mathbf{d}}}

\renewcommand{\u}{\ensuremath{\mathcal{U}}}
\newcommand{\ub}{\mathscr{U}}
\newcommand{\us}{\mathbf{U}}
\newcommand{\wus}{\widetilde{\mathbf{U}}}

\newcommand{\vb}{\mathscr{V}}
\newcommand{\vs}{\mathbf{V}}

\newcommand{\sa}{\mathbf{a}}
\renewcommand{\sb}{\mathbf{b}}
\newcommand{\sd}{\mathbf{d}}
\renewcommand{\sc}{\mathbf{c}}
\newcommand{\su}{\mathbf{u}}

\newcommand{\set}[1]{\left\{#1\right\}}
\newcommand{\la}{\lambda}

\newcommand{\ep}{\varepsilon}
\newcommand{\f}{\infty}
\newcommand{\de}{\alpha^*}

\newcommand{\al}{\alpha}
\newcommand{\lle}{\preccurlyeq}
\newcommand{\lge}{\succcurlyeq}
\newcommand{\si}{\sigma}

\newcommand{\ra}{\rightarrow}

\newcommand{\address}{Address: Department of Mathematics, University of North Texas, 1155 Union Circle \#311430, Denton, TX 76203-5017, USA; E-mail: allaart@unt.edu}

\title{An algebraic approach to entropy plateaus in non-integer base expansions}

\author{Pieter C. Allaart \footnote{\address}}

\date{\today}

\begin{document}

\maketitle

\begin{abstract}
For a positive integer $M$ and a real base $q\in(1,M+1]$, let $\u_q$ denote the set of numbers having a unique expansion in base $q$ over the alphabet $\{0,1,\dots,M\}$, and let $\us_q$ denote the corresponding set of sequences in $\{0,1,\dots,M\}^\N$. Komornik et al.~[{\em Adv. Math.} 305 (2017), 165--196] showed recently that the Hausdorff dimension of $\u_q$ is given by $h(\us_q)/\log q$, where $h(\us_q)$ denotes the topological entropy of $\us_q$. They furthermore showed that the function $H: q\mapsto h(\us_q)$ is continuous, nondecreasing and locally constant almost everywhere. The plateaus of $H$ were characterized by Alcaraz Barrera et al.~[{\em Trans. Amer. Math. Soc.}, {\bf 371} (2019), 3209--3258]. In this article we reinterpret the results of Alcaraz Barrera et al.~by introducing a notion of composition of fundamental words, and use this to obtain new information about the structure of the function $H$. This method furthermore leads to a more streamlined proof of their main theorem.

\bigskip
{\it AMS 2010 subject classification}: 11A63 (primary), 37B10, 37B40, 68R15 (secondary)

\bigskip
{\it Key words and phrases}: Beta-expansion, univoque set, topological entropy, entropy plateau, transitive subshift, composition of fundamental words.
\end{abstract}

\section{Introduction}

In recent years, there has been a great deal of interest in expansions of numbers in non-integer bases. Specifically, fix an integer $M\geq 1$, and for $q\in(1,M+1]$, consider expressions of the form
\begin{equation}
x=\sum_{i=1}^\f \frac{x_i}{q^i}=:\pi_q(x_1 x_2\dots),
\label{eq:q-expansion}
\end{equation}
where $x_i\in\{0,1,\dots,M\}$ for each $i$. The sequence $(x_i)$ is called a {\em $q$-expansion} of $x$; it exists if and only if $x\in J_{q,M}:=[0,M/(q-1)]$. Such non-integer base expansions were introduced by R\'enyi \cite{Renyi_1957} and studied further by Parry \cite{Parry_1960}. The subject seemed to be largely forgotten for about 30 years, until Erd\H{o}s and others \cite{Erdos_Horvath_Joo_1991,Erdos_Joo_1992,Erdos_Joo_Komornik_1990} breathed new life into it. These authors were interested in the set $\u_q$ of numbers in $J_{q,M}$ having a unique $q$-expansion, and in particular, in the set
\[
\ub:=\{q\in(1,M+1]: 1\in \u_q\}.
\]
Erd\H{o}s, Jo\'o and Komornik \cite{Erdos_Joo_Komornik_1990} showed for the case $M=1$ that $\u_q$ contains only the two endpoints $0$ and $1/(q-1)$ when $1<q\leq q_G:=(1+\sqrt{5})/2$.
Komornik and Loreti \cite{Komornik-Loreti-1998,Komornik-Loreti-2002} found the smallest element of $\ub$, which is now called the {\em Komornik-Loreti constant} and which we denote by $q_{KL}$; see below. Later, Glendinning and Sidorov \cite{Glendinning_Sidorov_2001} showed that $\u_q$ is countably infinite for $q_G<q<q_{KL}$; uncountable but of zero Hausdorff dimension for $q=q_{KL}$; and of positive Hausdorff dimension for $q>q_{KL}$. These results were generalized to larger alphabets by Baker \cite{Baker} and Kong, Li and Dekking \cite{Kong_Li_Dekking_2010}. Kong and Li \cite{Kong_Li_2015} further examined the Hausdorff dimension of $\u_q$, and more recently, Komornik, Kong and Li \cite{Komornik-Kong-Li-17} showed that this dimension is related to the topological entropy of the symbolic univoque set
\[
\us_q:=\{(x_i)\in\Omega_M: \pi_q((x_i))\in \u_q\},
\]
where $\Omega_M:=\{0,1,\dots,M\}^\N$. Namely,
\begin{equation}
\dim_H \u_q=\frac{h(\us_q)}{\log q}.
\label{eq:dimension-formula}
\end{equation}
Here the \emph{topological entropy} of any subset $X\subset\Omega_M$ is defined by
\[
h(X):=\lim_{n\ra\f}\frac{\log \# B_n(X)}{n},
\]
assuming the limit exists, where $B_n(X)$ denotes the set of all length $n$ prefixes of sequences from $X$, and ``$\log$" denotes the natural logarithm. (Since we will be considering different values of $M$ simultaneously, the ``neutral" base $e$ logarithm is used to avoid confusion.) Note that the limit always exists when $X$ is invariant under the left shift map $\sigma$, since the map $n\mapsto \#B_n(X)$ is then submultiplicative. This is the case, in particular, for $\us_q$.

Komornik, Kong and Li proved furthermore that the function 
$$H: q\mapsto h(\us_q)$$ 
is a devil's staircase. That is, $H$ is continuous, nondecreasing, and locally constant almost everywhere, and maps $(1,M+1]$ onto a nontrivial interval. Recently, the present author and Kong \cite{Allaart-Kong-2018a} discovered a gap in their proof, and gave a completely different demonstration of their results.
Kong and Li \cite{Kong_Li_2015} identified intervals on which $H$ is constant; their work was extended by Alcaraz Barrera et al.~\cite{AlcarazBarrera-Baker-Kong-2016}, who determined the {\em entropy plateaus} of $H$; that is, the maximal intervals of constancy of $H$. (This had already been done for the case $M=1$ by Alcaraz-Barrera \cite{Alcaraz_Barrera_2014}, who considered general symmetric subshifts of $\{0,1\}^\N$ and proved transitivity results and ergodic properties.) These entropy plateaus have since played an important role in the study of non-integer base expansions (e.g. \cite{Allaart-Baker-Kong-17,Allaart-Kong-2018b}). In this paper we detail the main results of \cite{AlcarazBarrera-Baker-Kong-2016}, and develop a new approach which not only presents these results in a new light, but also simplifies some of the proofs and provides new information on the nature of the entropy plateaus.

Before proceeding, we define some important notation. Let $\sigma: \Omega_M \to \Omega_M$ denote the left shift map, defined by $\sigma(d_1 d_2\dots):=(d_2 d_3\dots)$. Throughout the paper we will use the lexicographical ordering between sequences and words: For two sequences $(c_i), (d_i)\in\Omega_M$ we write $(c_i)\prec (d_i)$ or $(d_i)\succ (c_i)$ if there exists $n\in\N$ such that $c_1\ldots c_{n-1}=d_1\ldots d_{n-1}$ and $c_n<d_n$. Furthermore, we write $(c_i)\lle (d_i)$ if $(c_i)\prec (d_i)$ or $(c_i)=(d_i)$. Similarly, for two words $\c$ and $\d$ we say $\c\prec \d$ or $\d\succ\c$ if $\c0^\f\prec \d0^\f$.

Let $\al(q)=(\al_i(q))\in\Omega_M$ denote the \emph{quasi-greedy} $q$-expansion of $1$; that is, the lexicographically largest $q$-expansion of $1$ not ending with $0^\f$. The following useful characterization was proved in \cite[Proposition 2.3]{DeVries_Komornik_2009}:

\begin{lemma} \label{lem:quasi-greedy expansion-alpha-q}
The map $q\mapsto \al(q)$ is strictly increasing and bijective from $(1, M+1]$ to the set of sequences $(a_i)\in\Omega_M$ not ending with $0^\f$ and satisfying
\[
\si^n((a_i))\lle (a_i)\quad\forall n\ge 0.
\]
\end{lemma}

Let $(\tau_i)_{i=1}^\infty=1101\ 0011\ 00101101\dots$ be the (shifted) {\em Thue-Morse sequence}, defined by $\tau_i:=s_i\!\!\mod 2$, where $s_i$ is the sum of the digits in the binary representation of $i$. 
Recall from \cite{Komornik-Loreti-2002} that the Komornik-Loreti constant $q_{KL}=q_{KL}(M)$ satisfies
$\al(q_{KL})=(\lambda_i)$,
where for each $i\ge 1$,
\begin{equation} 
\label{eq:lambda-i}
\la_i=\la_i(M):=\begin{cases}
k+\tau_i-\tau_{i-1} & \qquad\textrm{if \quad$M=2k$},\\
k+\tau_i & \qquad\textrm{if \quad$M=2k+1$}.
\end{cases}
\end{equation}

For a word $c_1\dots c_k$ with $c_k<M$, we write $c_1\dots c_k^+:=c_1\dots c_{k-1}(c_k+1)$. Likewise, for a word $c_1\dots c_k$ with $c_k>0$, we write $c_1\dots c_k^-:=c_1\dots c_{k-1}(c_k-1)$. For any word $c_1\dots c_k$, its {\em reflection} is the word $\overline{c_1\dots c_k}:=(M-c_1)\dots(M-c_k)$. For an infinite sequence $(c_i)\in\Omega_M$, we similarly define $\overline{(c_i)_i}:=(M-c_i)_i$.
Using this notation, we define special bases $q_n'=q_n'(M)$ by
\[
\alpha(q_n')=\begin{cases}
\lambda_1\dots\lambda_{2^{n-1}}\big(\overline{\lambda_1\dots\lambda_{2^{n-1}}}^+\big)^\f & \mbox{if $M$ is even},\\
\lambda_1\dots\lambda_{2^n}\big(\overline{\lambda_1\dots\lambda_{2^n}}^+\big)^\f & \mbox{if $M$ is odd}.
\end{cases}
\]
We denote $q_1'$ also by $q_T$; it is called the {\em transitive base} in \cite{AlcarazBarrera-Baker-Kong-2016}. It can be deduced from Lemma \ref{lem:quasi-greedy expansion-alpha-q} that $q_n'$ is well defined; see \cite[Lemma 4.3]{AlcarazBarrera-Baker-Kong-2016} for the details. Finally, we define a special base $q_G$ (called a {\em generalized golden ratio} in \cite{Baker}) by
\begin{equation}
\alpha(q_G)=\begin{cases}
k^\infty & \mbox{if $M=2k$},\\
\big((k+1)k\big)^\infty & \mbox{if $M=2k+1$}.
\end{cases}
\label{eq:q_G}
\end{equation}
When $M=1$, we obtain $q_G=(1+\sqrt{5})/2$ as before. Let
\[
\vs:=\big\{(c_i)\in\Omega_M:\ \overline{(c_i)}\lle \sigma^n((c_i))\lle (c_i)\ \ \forall n\geq 0\big\}.
\]
In \cite{AlcarazBarrera-Baker-Kong-2016}, irreducible sequences in $\Omega_M$ were defined as follows:

\begin{definition}[\cite{AlcarazBarrera-Baker-Kong-2016}] \label{def:irreducible-sequence}
A sequence $(a_i)\in\vs$ is {\em irreducible} if for every $j\in\N$ the following implication holds:
\begin{equation}
\left[a_j>0\quad \mbox{and}\quad (a_1\dots a_j^-)^\f \in \vs\right] \qquad\Longrightarrow \qquad a_1\dots a_j (\overline{a_1\dots a_j}^+)^\f \prec (a_i).
\label{eq:irreducible-condition}
\end{equation}
A sequence $(a_i)\in\vs$ is {\em $*$-irreducible} if there exists $n\in\N$ such that $\alpha(q_{n+1}')\lle (a_i)\prec \alpha(q_n')$, and \eqref{eq:irreducible-condition} holds for every $j>2^n$ if $M$ is even, or for every $j>2^{n+1}$ if $M$ is odd.
\end{definition}

\begin{definition}[\cite{AlcarazBarrera-Baker-Kong-2016}] 
An interval $[p_L,p_R]\subset(q_{KL},M+1]$ is {\em irreducible} ({\em $*$-irreducible}) if $\alpha(p_L)$ is irreducible ($*$-irreducible) and there is a word $a_1\dots a_m$ with $a_m<M$ such that
\[
\alpha(p_L)=(a_1\dots a_m)^\f,  \qquad\mbox{and}\qquad \alpha(p_R)=a_1\dots a_m^+(\overline{a_1\dots a_m})^\f.
\]
\end{definition}

\begin{definition}
An interval $[p_L,p_R]$ is a {\em plateau of $H$}, or {\em entropy plateau}, if $[p_L,p_R]$ is a maximal interval (in the partial order of set inclusion) on which $H$ is positive and constant.
\end{definition}

The main result of \cite{AlcarazBarrera-Baker-Kong-2016} is:

\begin{theorem} \label{thm:entropy-plateaus}
An interval $[p_L,p_R]\subset (q_{KL},M+1]$ is an entropy plateau if and only if it is either irreducible or $*$-irreducible. Moreover, the irreducible entropy plateaus lie in $(q_T,M+1]$, and the $*$-irreducible entropy plateaus lie in $(q_{KL},q_T)$.
\end{theorem}

In this paper we present some new ideas that further illuminate the nature of the entropy plateaus, and in the process simplify the proof of Theorem \ref{thm:entropy-plateaus}. The paper is organized as follows. First, in Section \ref{sec:composition}, we introduce a family of intervals of which the entropy plateaus form a special subfamily. Each of these intervals is generated by a specific word $a_1\dots a_m$, which we call {\em fundamental}. We then define a notion of composition of fundamental words which allows us to put a semigroup structure on the set of all fundamental words in the special but important case $M=1$. This notion of composition yields a natural algebraic interpretation of the concepts of irreducible and $*$-irreducible intervals. We use it in Section \ref{sec:entropy} to uncover a direct connection between the $*$-irreducible intervals (for any $M$) and the irreducible intervals for $M=1$. This will provide new information about the behavior of the entropy function $H$ on the interval $(q_{KL},q_T]$. In particular, we derive the part of Theorem \ref{thm:entropy-plateaus} concerning $*$-irreducible plateaus in a very conceptual way from the part of the theorem concerning irreducible plateaus. For completeness, we include in Section \ref{sec:irreducible-plateaus} a more streamlined proof of the latter.

\section{Composition of fundamental words} \label{sec:composition}

This section borrows ideas and results from the paper \cite{Allaart-Kong-2018b} by the author and Kong. The interested reader is referred to that paper for the proofs.

\begin{definition} \label{def:fundamental-words}
A word $\sa=a_1\dots a_m\in\{0,1,\dots,M\}^m$ with $m\geq 2$ is {\em fundamental} if
\begin{equation*} 
\overline{a_1\ldots a_{m-i}}\lle a_{i+1}\ldots a_m\prec a_1\ldots a_{m-i}\quad\forall ~1\le i<m.
\end{equation*}
When $M\geq 2$, the ``word" $a_1\in\{0,1,\dots,M\}$ is {\em fundamental} if $\overline{a_1}\leq a_1<M$.
\end{definition}

An important observation is that $\sa^\f\in\vs$ for any fundamental word $\sa$. Note also that a fundamental word cannot end in the digit $M$.

\begin{remark}
There is a (near) one-to-one correspondence between fundamental words and the {\em primitive} words defined in \cite{AlcarazBarrera-Baker-Kong-2016}: a word $\sa$ is fundamental if and only if $\sa^+$ is primitive (with the exception of fundamental words of length $1$).
\end{remark}

Let $\A_M$ denote the set of all fundamental words with alphabet $\{0,1,\dots,M\}$.
For each $\sa\in\A_M$, there is by Lemma \ref{lem:quasi-greedy expansion-alpha-q} a unique interval $J_\sa=[q_L(\sa),q_R(\sa)]$ such that
\begin{equation}
\alpha(q_L(\sa))=\sa^\f, \qquad \alpha(q_R(\sa))=\sa^+(\overline{\sa})^\f.
\label{eq:q_L-and-q_R}
\end{equation}
We call $J_\sa$ a {\em fundamental interval generated by} the word $\sa$.
Observe for any $\sa\in\A_M$ that 
\begin{equation}
\overline{\sa^+}\prec\overline{\sa}\lle\sa\prec\sa^+.
\label{eq:label-ordering}
\end{equation}
The middle inequality is not always strict: $\overline{\sa}=\sa$ in the exceptional but important case that $M$ is even and $\sa=a_1=M/2$.

Figure \ref{fig1} shows a directed graph $G=(V,E)$ with edge set $E=\{e_0,e_1,\dots,e_4\}$ and with two edge labelings. In the figure, $\sa$ denotes a fundamental word. In view of \eqref{eq:label-ordering}, the labeled graph $\mathcal G_\sa=(G, \mathcal L_\sa)$ with labeling $\mathcal L_\sa: E\to L_\sa:=\big\{\sa, \sa^+, \overline{\sa}, \overline{\sa^+}\big\}$ is right-resolving, i.e. the out-going edges from the same vertex in $\mathcal G_\sa$ have different labels. Let $X_\sa$ be the set of infinite sequences generated by the automata $\mathcal G_\sa=(G, \mathcal L_\sa)$, beginning at the ``Start" vertex (cf.~\cite{Lind_Marcus_1995}).  We emphasize that each digit $\sd$ in $L_\sa$ is a block of length $|\sa|$, and any sequence in $X_\sa$ is an infinite concatenation of blocks from $L_\sa$. We also let $X_\sa^{fin}$ denote the set of all finite prefixes of sequences in $X_\sa$ whose length is a multiple of $|\sa|$.

Likewise, the labeled graph $\mathcal G^*=(G, \mathcal L^*)$ with labeling $\mathcal L^*:E \to\set{0,1}$ is right-resolving. Let $X^*\subset \{0,1\}^\N$ be the set of all infinite sequences generated by the automata $\mathcal G^*$, and note that $X^*=\{(x_i)\in\{0,1\}^\N:x_1=1\}$. Thus, for each $q\in(1,2]$ the quasi-greedy expansion $\de(q)$ of $1$ in base $q$ is an element of $X^*$, so that $\alpha^*(q)$ is the sequence of labels on an infinite path generated by the automata $\mathcal{G}^*$. Since the labeling $\mathcal{L}^*$ is right-resolving, this path is unique.


\begin{figure}[ht]
  \centering
  \begin{tikzpicture}[->,>=stealth',shorten >=1pt,auto,node distance=4cm,
                    semithick]

  \tikzstyle{every state}=[minimum size=0pt,fill=none,draw=black,text=black]

  \node[state] (A)                    { $A$};
  \node[state]         (B) [ right of=A] {$B$ };
  \node[state]         (C) [ above of=A,yshift=-1cm] {$Start$};

  \path[->,every loop/.style={min distance=0mm, looseness=40}]
  (C) edge[->,left] node{$e_0:\; \sa^+\; / \; 1$} (A)
  
   (A) edge [loop left,->]  node {$e_4:\; \overline{\sa}\;/\; 1$} (A)
            edge  [bend left]   node {$e_1:\; \overline{\sa^+}\;/\; 0$} (B)

   (B) edge [loop right] node {$e_2:\; \sa\; /\; 0$} (B)
            edge  [bend left] node {$e_3:\; \sa^+\; / \; 1$} (A);
\end{tikzpicture}
\caption{The labeled graph $\mathcal G_\sa=(G, \mathcal L_\sa)$ with labeling $\mathcal L_\sa: E\to L_\sa:=\big\{\sa, \sa^+, \overline{\sa}, \overline{\sa^+}\big\}$, and the labeled graph $\mathcal G^*=(G, \mathcal L^*)$ with labeling $\mathcal L^*: E\to\set{0, 1}$.}
\label{fig1}
\end{figure}

For a given word $\sa\in\A_M$, 
we define a map $\Phi_\sa: X_\sa\to X^*$ as follows. Given a sequence $(\sd_i)\in X_\sa$ of blocks from $L_\sa$, there is a unique infinite path $e_{i_1}e_{i_2}\dots$ in $G$ with $i_1=0$ such that $\mathcal{L}_\sa(e_{i_1}e_{i_2}\dots)=\sd_1\sd_2\dots$. Define $\Phi_\sa((\sd_i))$ by
\[
\Phi_\sa(\sd_1\sd_2\dots):=\mathcal{L}^*(e_{i_1}e_{i_2}\dots).
\]
Similarly, for any finite word $\sd_1\dots\sd_k$ starting with $\sd_1=\sa^+$ generated by the labeling $\mathcal{L}_\sa$ along a path $e_{i_1}\dots e_{i_k}$ with $i_1=0$, we define $\Phi_\sa(\sd_1\dots\sd_k):=\mathcal{L}^*(e_{i_1}\dots e_{i_k})$. Likewise, for a word $\sd_1\dots\sd_k$ starting with $\sd_1=\overline{\sa^+}$ generated by the labeling $\mathcal{L}_\sa$ along a path $e_{i_1}\dots e_{i_k}$ (necessarily with $i_1=1$), we define $\Phi_\sa(\sd_1\dots\sd_k)$ in the same way. When $\sa\neq\overline{\sa}$, we can also define $\Phi_\sa(\sd_1\dots\sd_k)$ for words $\sd_1\dots \sd_k$ starting with $\sd_1=\sa$ or $\overline{\sa}$ in the same way. However, we leave $\Phi_\sa(\sd_1\dots\sd_k)$ undefined for such words when $\sa=\overline{\sa}$; the reason is that in this case, the edges $e_2$ and $e_4$ have the same label under $\mathcal{L}_\sa$, but different labels under $\mathcal{L}^*$, causing some ambiguity.


\begin{remark}
The definition of $\Phi_\sa$ slightly generalizes the definition of the maps $\Phi_J$ introduced in \cite[Section 3]{Allaart-Kong-2018b}. But our definitions here require a bit more care due to the possibility that $\sa=\overline{\sa}$. Nonetheless, since the graph $\mathcal{G}_\sa$ is right-resolving even then, all of the proofs from \cite[Section 3]{Allaart-Kong-2018b} go through in the same way. We therefore omit them here.
\end{remark}

\begin{remark}
As pointed out by a referee, the map $\Phi_\sa$ is somewhat reminiscent of the renormalization of kneading invariants of Lorenz maps in \cite{Glendinning_Hall_1996}, though the author did not see a direct connection.
\end{remark}

\begin{lemma} \label{lem:again-fundamental}
If $\sc\in X_\sa^{fin}$ is fundamental, then $\Phi_\sa(\sc)$ is fundamental.
\end{lemma}

\begin{proof}
This follows as in the proof of \cite[Lemma 3.7]{Allaart-Kong-2018b}.
\end{proof}

Going forward, superscript $^*$ indicates that the alphabet $\{0,1\}$ is intended, i.e. $M=1$.

In \cite[Proposition 3.5]{Allaart-Kong-2018b} it is shown that $\Phi_\sa: X_\sa\to X^*$ is a strictly increasing bijection. Let 
\[
\vb:=\{q\in(1,M+1]: \alpha(q)\in\vs\},
\]
and $\vb^*:=\vb$ when $M=1$. By \cite[Proposition 3.4]{Allaart-Kong-2018b}, $\alpha(q)\in X_{\sa}$ for $q\in (q_L(\sa),q_R(\sa)]\cap\vb$, where $q_L(\sa)$ and $q_R(\sa)$ were defined in \eqref{eq:q_L-and-q_R}. Hence, $\Phi_\sa$ induces a map
\[
\hat{\Phi}_\sa: (q_L(\sa),q_R(\sa)] \cap \vb \to \vb^*, \qquad q\mapsto (\al^*)^{-1}\circ\Phi_\sa\circ\al(q).
\]
The map $\hat{\Phi}_\sa$ is an increasing homeomorphism; see \cite[Proposition 3.8]{Allaart-Kong-2018b}.

\begin{definition} \label{def:fundamental-word-composition-M}
For $\sa\in\A_M$ and $\sb\in\A_1$ we define the {\em composition} $\sa\circ\sb$ by
\[
\sa\circ\sb:=\Phi_\sa^{-1}(\sb).
\]
\end{definition}

Note that this is well defined since $\sb$ is a prefix of a sequence in $X^*$. The composition $\sa\circ\sb$ should not be confused with the concatenation $\sa\sb$. Indeed, $|\sa\circ\sb|=|\sa||\sb|$, whereas $|\sa\sb|=|\sa|+|\sb|$.
Observe that the composition of two words in $\A_M$ is not defined when $M\geq 2$.

\begin{example} \label{ex:composition}
Let $M=1$ and take $\sa=10$, $\sb=110$. Then:
\begin{enumerate}
\item $\sa\circ\sb=\Phi_\sa^{-1}(110)=\sa^+\overline{\sa}\overline{\sa^+}=110100$;
\item $\sb\circ\sa=\Phi_\sb^{-1}(10)=\sb^+\overline{\sb^+}=111000$.
\end{enumerate}
\end{example}

\begin{example}
Let $M=2$ and take $\sa=1, \sb=1110$. Then $\sa\in \A_2$, $\sb\in\A_1$, and $\sa\circ\sb=\sa^+\overline{\sa}^2\overline{\sa^+}=2110$.
\end{example}

\begin{lemma} \label{lem:closed-under-composition}
If $\sa\in \A_M$ and $\sb\in\A_1$, then $\sa\circ\sb\in\A_M$.
\end{lemma}

\begin{proof}
Let $\sa=a_1\dots a_m$, $\sb=b_1\dots b_n$, and put $\sc:=c_1\dots c_{mn}=\sa\circ\sb$. We must verify that 
\begin{equation}
\overline{c_1\ldots c_{mn-i}}\lle c_{i+1}\ldots c_{mn}\prec c_1\ldots c_{mn-i}\quad\forall ~1\le i<mn.
\label{eq:c-fundamental}
\end{equation}
Since $\sa\in\A_M$, by Definition \ref{def:fundamental-words} it follows that
\begin{equation*}
\overline{a_1\ldots a_m^+}\prec a_{i+1}\ldots a_m^+\overline{a_1\ldots a_i}\prec a_1\ldots a_m^+
\end{equation*}
and
\begin{equation*}
\overline{a_1\ldots a_m^+}\prec a_{i+1}\ldots a_m a_1\ldots a_i\prec a_1\ldots a_m^+
\end{equation*}
for all $1\le i<m$. Note that we can write $\sc=\sd_1\dots\sd_n$, where $\sd_i\in \big\{\sa,\sa^+,\overline{\sa},\overline{\sa^+}\big\}$ for each $i$. Since the block $\sa^+$ can only be followed by $\overline{\sa}$ or $\overline{\sa^+}$ and the block $\sa$ can only be followed by $\sa$ or $\sa^+$, it follows from the above inequalities that \eqref{eq:c-fundamental} holds whenever $i$ is not a multiple of $m$. So it remains to verify that
\begin{equation}
\overline{\sd_1\dots \sd_{n-j}}\lle \sd_{j+1}\dots \sd_n\prec \sd_1\dots \sd_{n-j}\quad\forall ~1\le j<n.
\label{eq:d-block-fundamental}
\end{equation}
But this follows from the admissibility of $\sb$, which implies
\[
\overline{b_1\dots b_{n-j}}\lle b_{j+1}\dots b_n\prec b_1\dots b_{n-j} \quad\forall ~1\le j<n.
\]
For instance, since $\sd_1=\sa^+$, the second inequality in \eqref{eq:d-block-fundamental} is obvious if $\sd_{j+1}\prec \sa^+$, so let us assume $\sd_{j+1}=\sa^+$. Then $\sd_{j+1}\dots \sd_n$ is a prefix of a sequence in $X_\sa$, so, since $\Phi_\sa$ is strictly increasing on $X_\sa$,
\begin{equation*}
\sd_{j+1}\dots \sd_n=\Phi_\sa^{-1}(b_{j+1}\dots b_n)\prec \Phi_\sa^{-1}(b_1\dots b_{n-j})
=\sd_1\dots \sd_{n-j}.
\end{equation*}
The first inequality in \eqref{eq:d-block-fundamental} follows similarly, by taking reflections.
\end{proof}

In view of the last lemma, the interval $J_{\sa\circ\sb}=[q_L(\sa\circ\sb),q_R(\sa\circ\sb)]$ is well defined, and it is easy to see that $J_{\sa\circ\sb}\subset J_\sa$.

\begin{lemma} \label{lem:compositions}
For $\sa\in\A_M$ and $\sb\in \A_1$, we have
\[
\Phi_{\sa\circ\sb}=\Phi_{\sb}\circ \Phi_{\sa}|_{X_{\sa\circ\sb}}, \qquad \hat{\Phi}_{\sa\circ\sb}=\hat{\Phi}_{\sb}\circ \hat{\Phi}_{\sa}|_{\alpha^{-1}(X_{\sa\circ\sb})},
\]
and as a result,
\[
\Phi_{\sa\circ\sb}^{-1}=\Phi_{\sa}^{-1}\circ \Phi_{\sb}^{-1}, \qquad \hat{\Phi}_{\sa\circ\sb}^{-1}=\hat{\Phi}_{\sa}^{-1}\circ \hat{\Phi}_{\sb}^{-1}.
\]
\end{lemma}

\begin{proof}
If $\sc\in X_\sa^{fin}$ is fundamental, then 
\[
\Phi_\sa(\sc^+)=\big(\Phi_\sa(\sc)\big)^+, \qquad \Phi_\sa(\overline{\sc})=\overline{\Phi_\sa(\sc)}, \qquad 
\Phi_\sa\big(\overline{\sc^+}\big)=\overline{\Phi_\sa(\sc^+)},
\]
as can be seen quickly from the definition of $\Phi_\sa$. In particular, setting $\sc=\sa\circ\sb$, we have
\begin{align}
\begin{split}
\Phi_\sa(\sa\circ\sb)&=\Phi_\sa(\Phi_\sa^{-1}(\sb))=\sb,\\
\Phi_\sa\big((\sa\circ\sb)^+\big)&=\big(\Phi_\sa(\sa\circ\sb)\big)^+=\sb^+,\\
\Phi_\sa\big(\overline{\sa\circ\sb}\big)&=\overline{\Phi_\sa(\sa\circ\sb)}=\overline{\sb},\\
\Phi_\sa\big(\overline{(\sa\circ\sb)^+}\big)&=\overline{\Phi_\sa\big((\sa\circ\sb)^+\big)}=\overline{\sb^+}.
\end{split}
\label{eq:four-identities}
\end{align}
Now let $\sd$ be one of the blocks $\sa\circ\sb, (\sa\circ\sb)^+, \overline{\sa\circ\sb}$ or $\overline{(\sa\circ\sb)^+}$. 
Since $\sb\in \A_1$, we have $\sb\neq\overline{\sb}$ so all four labels in $L_\sb:=\{\sb,\sb^+,\overline{\sb},\overline{\sb^+}\}$ are distinct.
By the four identities above, $\Phi_\sa(\sd)\in L_\sb$, and therefore $\Phi_\sb(\Phi_\sa(\sd))$ is well defined.
Note that $\Phi_{\sa\circ\sb}(\sd)=1$ if and only if $\sd=(\sa\circ\sb)^+$ or $\overline{\sa\circ\sb}$. Recalling the definition of $\Phi_\sb$ and the identities \eqref{eq:four-identities}, it follows that $\Phi_{\sa\circ\sb}(\sd)=\Phi_\sb(\Phi_\sa(\sd))$. This implies that $\Phi_{\sa\circ\sb}=\Phi_{\sb}\circ \Phi_{\sa}|_{X_{\sa\circ\sb}}$. The corresponding identity for $\hat{\Phi}_{\sa\circ\sb}$ follows since
\begin{align*}
\hat{\Phi}_\sb\circ\hat{\Phi}_\sa|_{\alpha^{-1}(X_{\sa\circ\sb})} &=\big((\alpha^*)^{-1}\circ\Phi_\sb\circ\alpha^*\big)\circ\big((\alpha^*)^{-1}\circ\Phi_\sa|_{X_{\sa\circ\sb}}\circ\alpha\big)\\
&=(\alpha^*)^{-1}\circ\big(\Phi_\sb\circ \Phi_\sa|_{X_{\sa\circ\sb}}\big)\circ\alpha\\
&=(\alpha^*)^{-1}\circ \Phi_{\sa\circ\sb}\circ\alpha=\hat{\Phi}_{\sa\circ\sb}.
\end{align*}
The identities for the inverse maps are a direct consequence of the ones already proved, noting that the composition $\Phi_{\sa}^{-1}\circ \Phi_{\sb}^{-1}$ is well defined since $\Phi_\sb^{-1}: X^*\to X_\sb\subset X^*$ and $\Phi_\sa^{-1}: X^*\to X_\sa$.
\end{proof}

\begin{proposition} \label{prop:fundamental-semigroup}
The set $\A_1$ with the operation $\circ$ is a non-Abelian semigroup. That is,
\begin{enumerate}[(i)]
\item If $\sa,\sb\in\A_1$, then $\sa\circ\sb\in\A_1$. 
\item For any three words $\sa,\sb,\sc\in\A_1$,
$(\sa\circ \sb)\circ \sc=\sa\circ(\sb\circ \sc)$.
\end{enumerate}
\end{proposition}

\begin{proof}
That $\A_1$ is closed under $\circ$ is the content of Lemma \ref{lem:closed-under-composition}.
Associativity follows from Lemma \ref{lem:compositions}, since
\[
(\sa\circ\sb)\circ\sc=\Phi_{\sa\circ\sb}^{-1}(\sc)=\Phi_\sa^{-1}(\Phi_\sb^{-1}(\sc))=\Phi_\sa^{-1}(\sb\circ\sc)=\sa\circ(\sb\circ\sc).
\]
That $(\A_1,\circ)$ is non-Abelian can be seen from Example \ref{ex:composition}.
\end{proof}

\begin{remark}
Note by Lemma \ref{lem:compositions} that the maps $\sa\mapsto \Phi_\sa^{-1}$ and $\sa\mapsto \hat{\Phi}_\sa^{-1}$ are semigroup homomorphisms when $M=1$.
\end{remark}

\begin{lemma} \label{lem:irreducible-implies-maximal}
Let $\sa$ and $\sb$ be distinct fundamental words in $\A_M$ such that $J_\sb\subset J_\sa$. Then there is a word $\sc\in\A_1$ such that $\sb=\sa\circ\sc$.
\end{lemma}

\begin{proof}
By \cite[Lemma 3.2]{Allaart-Kong-2018b}, $\sb$ is a prefix of a sequence in $X_\sa$, so we can take $\sc=\Phi_\sa(\sb)$. Then $\sc\in\A_1$ by Lemma \ref{lem:again-fundamental}, and $\sa\circ\sc=\Phi_\sa^{-1}(\sc)=\sb$.
\end{proof}

\begin{definition}
The {\em unit lift} is the word $\su=\su_M$ defined as follows:
\[
\su=\su_M=\begin{cases}
k & \mbox{if $M=2k$},\\
(k+1)k & \mbox{if $M=2k+1$}.
\end{cases}
\]
\end{definition}

Observe that $\su\in\A_M$ and $J_\su=[q_G,q_T]$.

\begin{definition} \label{def:irreducible-word}
Let $\sa\in\A_M$.
\begin{enumerate}[(i)]
\item Say $\sa$ is {\em irreducible} if there do not exist $\sc\in\A_M$ and $\sd\in\A_1$ such that $\sa=\sc\circ\sd$. 
\item Say $\sa$ is {\em $n$-irreducible} if $\sa=\su\circ (10)^{\circ (n-1)}\circ\sb$ for some irreducible word $\sb\in\A_1$.
\item Say $\sa$ is {\em $*$-irreducible} if $\sa$ is $n$-irreducible for some $n\in\N$.
\end{enumerate}
Here $(10)^{\circ k}$ means the $k$-fold composition of $10$ with itself.
\end{definition}

\begin{definition} \label{def:irreducible-interval}
Call the interval $J_\sa$ {\em irreducible} ({\em $n$-irreducible, $*$-irreducible}) if $\sa$ is irreducible ($n$-irreducible, $*$-irreducible).
\end{definition}

\begin{proposition}
Let $\sa\in\A_M$. Then $\sa$ has a unique decomposition into irreducible words. That is, there is a unique $k\in\N$ and a unique $k$-tuple $(\sc_1,\dots,\sc_k)$ with $\sc_1\in\A_M$, $\sc_2,\dots,\sc_k\in\A_1$, and each $\sc_i$ irreducible, such that
$\sa=\sc_1\circ \sc_2\circ\dots\circ \sc_k$.
\end{proposition}

\begin{proof}
The existence of such a decomposition follows directly from Definition \ref{def:irreducible-word}(i). Uniqueness follows from the fact that $J_{\sa\circ\sb}\subset J_\sa$, plus induction, since Lemma \ref{lem:irreducible-implies-maximal} implies that any two distinct irreducible intervals are disjoint.
\end{proof}

\begin{lemma} \label{lem:irreducible-characterization}
Let $q\in[q_G,M+1]$. The sequence $\alpha(q)$ is irreducible if and only if $q\not\in(q_L,q_R]$ for any fundamental interval $[q_L,q_R]$.
\end{lemma}

\begin{proof}
Write $\alpha(q)=a_1 a_2\dots$.
Assume first that $\alpha(q)$ is irreducible, and suppose $q\in(q_L(\sb),q_R(\sb)]$ for some fundamental word $\sb=b_1\dots b_j$. Then
\begin{equation}
(b_1\dots b_j)^\f=\alpha(q_L(\sb)) \prec \alpha(q)\lle \alpha(q_R(\sb))=b_1\dots b_j^+(\overline{b_1\dots b_j})^\f.
\label{eq:b-interval}
\end{equation}
So $\alpha(q)$ begins with $b_1\dots b_j^+$, i.e. $a_1\dots a_j=b_1\dots b_j^+$. But then $a_1\dots a_j^-=b_1\dots b_j$ and so $(a_1\dots a_j^-)^\f\in\vs$. In view of \eqref{eq:b-interval}, this contradicts \eqref{eq:irreducible-condition}.

Conversely, suppose $q\not\in(q_L,q_R]$ for any fundamental interval $[q_L,q_R]$. We first verify that $\alpha(q)\in\vs$. Either $q=q_G$ and $\alpha(q)\in\vs$ by \eqref{eq:q_G}; or else $q>q_T>q_{KL}$. In the second case, $q\in\overline{\ub}$, because each connected component of $(q_{KL},M+1]\backslash \overline{\ub}$ is contained in the interior of a fundamental interval (see \cite{Kong_Li_2015}). But $\overline{\ub}\subset\vb$, so $q\in\vb$ and hence $\alpha(q)\in\vs$.

Now let $j\in\N$ be such that $(a_1\dots a_j^-)^\f \in \vs$. Consider the fundamental interval $[q_L,q_R]$ given by
\[
\alpha(q_L)=(a_1\dots a_j^-)^\f, \qquad \alpha(q_R)=a_1\dots a_j(\overline{a_1\dots a_j}^+)^\f.
\]
Note that $q>q_L$ since $\alpha(q)\succ \alpha(q_L)$, so by our hypothesis, $q>q_R$ and so $\alpha(q)\succ \alpha(q_R)$. But this is equivalent to the consequent of \eqref{eq:irreducible-condition}. Hence, $\alpha(q)$ is irreducible.
\end{proof}

\begin{proposition} \label{lem:irreducible-connection}
Let $\sa\in\A_M$. Then $\sa^\f$ is irreducible (in the sense of Definition \ref{def:irreducible-sequence}) if and only if $\sa$ is irreducible (in the sense of Definition \ref{def:irreducible-word}).
\end{proposition}

\begin{proof}
Let $q$ be the base with $\alpha(q)=\sa^\f$; then $q$ is the left endpoint of $J_\sa$.
Suppose first that $\sa$ is irreducible. Then by Lemma \ref{lem:irreducible-implies-maximal}, $J_\sa$ is a maximal fundamental interval, so there is no fundamental interval $[q_L,q_R]$ such that $q\in (q_L,q_R]$. Hence, by Lemma \ref{lem:irreducible-characterization}, $\sa^\f=\alpha(q)$ is irreducible.

Conversely, suppose $\sa^\f$ is irreducible. By Lemma \ref{lem:irreducible-characterization}, there is no fundamental interval $[q_L,q_R]$ such that $q\in (q_L,q_R]$. Hence $J_\sa$ is a maximal fundamental interval. This implies $\sa$ is irreducible.
\end{proof}

\begin{remark}
Proposition \ref{lem:irreducible-connection} provides a natural algebraic interpretation of the notion of ``irreducible" from Definition \ref{def:irreducible-sequence}: $\sa^\f$ is irreducible if and only if $\sa$ can not be written as the composition of two or more fundamental words.
\end{remark}

\section{Irreducible intervals and entropy plateaus} \label{sec:entropy}

Now if we set $\sa_n:=\su_M\circ (10)^{\circ (n-1)}$ for $n\in\N$, then $\sa_n\in\A_M$ and the interval $J_{\sa_n}=[q_L(\sa_n),q_R(\sa_n)]$ has right endpoint $q_R(\sa_n)=q_n'$. Note that $q_n'\searrow q_{KL}$, and $\sa_n\to (\lambda_i)$. Also set $q_0':=M+1$, and define the intervals
\[
I_n:=(q_{n+1}',q_n'], \qquad n=0,1,2,\dots,
\]
so $\bigcup_{n=0}^\f I_n=(q_{KL},M+1]$, with the union disjoint. When $M=1$ we write $I_n^*:=I_n$.

Before stating the first lemma of this section, we recall that $J_\su=[q_G,q_T]$, that
\[
q_G<q_{KL}<\dots<q_{n+1}'<q_n'<\dots<q_2'<q_1'=q_T,
\]
and all these bases belong to $\vb$.

\begin{lemma} \label{lem:Phi-hat-properties}
\begin{enumerate}[(i)]
\item $\hat{\Phi}_{\su}$ is increasing and maps $[q_G,q_T]\cap\vb$ bijectively onto $\vb^*$. 
\item $\hat{\Phi}_{\su}(q_{KL})=q_{KL}^*$.
\item $\hat{\Phi}_{\su}(I_{n+1}\cap\vb)=I_n^*\cap\vb^*$ for all $n\geq 0$.
\end{enumerate}
\end{lemma}

\begin{proof}
The proof of (i) is the same as the proof of \cite[Proposition 3.5]{Allaart-Kong-2018b}. For (ii) we must verify that $\Phi_{\su}((\lambda_i))=(\tau_i)$. Consider first the case when $M$ is even. We prove (ii) for $M=2$; the proof for other even values of $M$ is the same modulo a renaming of the digits. We will show inductively that
\begin{equation}
\Phi_\su(\lambda_1\dots\lambda_{2^k})=\tau_1\dots\tau_{2^k}, \qquad\mbox{for all $k\geq 1$},
\label{eq:finite-1-to-1-map}
\end{equation}
where $\su=1$ and $\lambda_i:=\lambda_i(2)$. For $k=1$, we get $\Phi_\su(\lambda_1\lambda_2)=\Phi_\su(21)=\Phi_\su(\su^+\overline{\su})=11=\tau_1\tau_2$. Assuming \eqref{eq:finite-1-to-1-map} holds for some arbitrary $k\in\N$, we use the relationship
\[
\lambda_{2^k+1}\dots\lambda_{2^{k+1}}=\overline{\lambda_1\dots\lambda_{2^k}}^+
\]
to obtain
\begin{align*}
\Phi_\su(\lambda_1\dots\lambda_{2^{k+1}})&=\Phi_\su(\lambda_1\dots\lambda_{2^k})\Phi_\su(\lambda_{2^k+1}\dots\lambda_{2^{k+1}})\\
&=\Phi_\su(\lambda_1\dots\lambda_{2^k})\Phi_\su(\overline{\lambda_1\dots\lambda_{2^k}}^+)\\
&=\Phi_\su(\lambda_1\dots\lambda_{2^k})\overline{\Phi_\su(\lambda_1\dots\lambda_{2^k})}^+\\
&=\tau_1\dots\tau_{2^k}\overline{\tau_1\dots\tau_{2^k}}^+\\
&=\tau_1\dots\tau_{2^{k+1}}.
\end{align*}
Here the first equality is justified by the fact that $\lambda_{2^k+1}=\overline{\lambda_1}=0=\overline{\su^+}$, so the expressions on the right side are well defined.

Consider next the case when $M$ is odd. We prove (ii) for $M=1$; the proof for other odd values of $M$ is the same modulo a renaming of the digits. Here we must prove that $\Phi_{\su}((\tau_i))=(\tau_i)$, where $\su=10$. We do this by induction, by showing that
\begin{equation}
\Phi_{\su}(\tau_1\dots\tau_{2^k})=\tau_1\dots\tau_{2^{k-1}}, \qquad\mbox{for all $k\geq 1$}.
\label{eq:finite-2-to-1-map}
\end{equation}
Note first that $\Phi_{\su}(\tau_1\tau_2)=\Phi_{\su}(\su^+)=1=\tau_1$, so \eqref{eq:finite-2-to-1-map} holds for $k=1$. The induction step proceeds essentially the same way as for the case $M=2$ above, so we omit the details.


To prove (iii), it is sufficient to show that $\hat{\Phi}_{\su}(q_{n+1}'(M))=q_n'(1)$ for all $n\geq 0$. This means that for even $M$ we have to show
\[
\Phi_\su\left(\lambda_1\dots\lambda_{2^n}\big(\overline{\lambda_1\dots\lambda_{2^n}}^+\big)^\f\right)=\tau_1\dots\tau_{2^n}(\overline{\tau_1\dots\tau_{2^n}}^+)^\f \qquad\forall n\in\N,
\]
and for odd $M$ we have to show
\[
\Phi_\su\left(\lambda_1\dots\lambda_{2^{n+1}}\big(\overline{\lambda_1\dots\lambda_{2^{n+1}}}^+\big)^\f\right)=\tau_1\dots\tau_{2^n}(\overline{\tau_1\dots\tau_{2^n}}^+)^\f \qquad\forall n\in\N.
\]
These equations follow in the same way as in the proof of (ii) above.
\end{proof}

\begin{proposition} \label{prop:location-of-n-irreducible-intervals}
All irreducible intervals, except $[q_G,q_T]$, lie in $I_0=(q_T,M+1]$. All $n$-irreducible intervals, except the one with right endpoint $q_{n+1}'$, lie in $I_n$, for $n\in\N$.
\end{proposition}

\begin{proof}
This follows directly from Definition \ref{def:irreducible-word} and Lemma \ref{lem:Phi-hat-properties}.
\end{proof}

Having established the action of the map $\hat{\Phi}_\su$, we now turn to investigate the entropy function
$H: q\mapsto h(\us_q)$.
We recall the following lexicographical characterization of $\us_q$ (see \cite{DeVries_Komornik_2009}): $(d_i)\in\us_q$ if and only if $(d_i)\in\Omega_M$ satisfies
\begin{equation}\label{eq:characterization-unique expansion}
\begin{split}
d_{n+1}d_{n+2}\ldots&\prec \al(q)\qquad\textrm{if}\quad d_n<M,\\
d_{n+1}d_{n+2}\ldots&\succ\overline{\al(q)}\qquad\textrm{if}\quad d_n>0.
\end{split}
\end{equation}
Motivated by this characterization, we define the simpler set
\begin{equation*}
\wus_q:=\set{(d_i)\in\Omega_M: \overline{\al(q)}\prec\sigma^n((d_i))\prec \al(q)\ \forall n\ge 0}.
\end{equation*}
Note that $h(\us_q)=h(\wus_q)$ (see \cite{Komornik-Kong-Li-17}). Write $H^*(q):=h(\us_q^*)$, and let 
\[
c_M:=\begin{cases}
1 & \mbox{if $M$ is even},\\
\frac12 & \mbox{if $M$ is odd}.
\end{cases}
\]

\begin{proposition} \label{prop:entropy-bridge}
Let $q\in(q_G,q_T]\cap\vb$, and put $\hat{q}:=\hat{\Phi}_{\su}(q)$. Then
\begin{enumerate}[(i)]
\item 
$
\Phi_{\su}\big(\big\{(x_i)\in \wus_q: x_1\dots x_{|\su|}=\su^+\big\}\big)=\big\{(y_i)\in\wus_{\hat{q}}^*: y_1=1\big\}.
$
\item $H(q)=c_M H^*(\hat{q})$.
\end{enumerate}
\end{proposition}

\begin{proof}
Statement (i) is a special case of \cite[Proposition 3.8(iii)]{Allaart-Kong-2018b};
(ii) will follow from (i) once we show that
\begin{equation}
h(\wus_{\hat q}^*)=h\big(\big\{(y_i)\in\wus_{\hat{q}}^*: y_1=1\big\}\big)
\label{eq:first-one}
\end{equation}
and
\begin{equation}
h(\wus_q)=h\big(\big\{(x_i)\in \wus_q: x_1 \dots x_{|\su|}=\su^+\big\}\big),
\label{eq:first-two-ones}
\end{equation}
using that $\Phi_{\su}$ is a $|\su|$-block map.

The set $\big\{(y_i)\in\wus_{\hat{q}}^*: y_1=1\big\}$ and its reflection, $\big\{(y_i)\in\wus_{\hat{q}}^*: y_1=0\big\}$, together make up $\wus_{\hat{q}}^*$. This gives \eqref{eq:first-one}.

The proof of \eqref{eq:first-two-ones} is slightly more involved. Assume first that $M$ is odd, and without loss of generality suppose $M=1$. Then $\su^+=11$. Let $A_{kl}:=\big\{(x_i)\in \wus_q: x_1 x_2=kl\big\}$ for $k,l=0,1$. Since $\wus_q=A_{11}\cup A_{10}\cup A_{01}\cup A_{00}$, it follows by symmetry that
\[
\#B_n(\wus_q)=2\#B_n(A_{11})+2\#B_n(A_{10}).
\]
Every sequence in $A_{10}$ is either $(10)^\infty$, or of the form $(10)^j x_1 x_2\dots$ with $(x_i)\in A_{11}$, or of the form $(10)^j 1 x_1 x_2\dots$ with $(x_i)\in A_{00}$, for some $j\geq 0$. Hence 
\[
\#B_n(A_{10})\leq \sum_{j=0}^{n-1} \#B_j(A_{11})\leq n\#B_n(A_{11}),
\]
so that $\#B_n(A_{11})\leq\#B_n(\wus_q)\leq (2n+2)\#B_n(A_{11})$. Taking logarithms, dividing by $n$ and letting $n\to\infty$, it follows that $h(A_{11})=h(\wus_q)$, proving \eqref{eq:first-two-ones}.


If $M$ is even, then $|\su|=1$ and the proof of \eqref{eq:first-two-ones} is basically the same as that of \eqref{eq:first-one}, since any sequence in $\wus_q$ begins with either $\su$ or $\overline{\su}$.
\end{proof}


Applying Proposition \ref{prop:entropy-bridge} repeatedly, we obtain Lemma 5.4 of \cite{AlcarazBarrera-Baker-Kong-2016}:

\begin{corollary} \label{cor:entropy-of-qT}
For each $n\in\N$ we have
\[
H(q_n')=\frac{c_M\log 2}{2^{n-1}}.
\]
In particular, $H(q_T)=c_M\log 2$.
\end{corollary}


\begin{corollary}
Let $q\in[q_G,q_T]\cap\vb$, and put $\hat{q}:=\hat{\Phi}_{\su}(q)$. Then
\[
\dim_H \u_q=c_M\frac{\log\hat{q}}{\log q}\dim_H \u_{\hat q}^*.
\]
\end{corollary}

\begin{proof}
This follows at once from Proposition \ref{prop:entropy-bridge} and \eqref{eq:dimension-formula}.
\end{proof}



\begin{lemma} \label{lem:dense-intervals}
\begin{enumerate}[(i)]
\item The union of all irreducible intervals is dense in $I_0=[q_T,M+1]$.
\item For each $n\in\N$, the union of all $n$-irreducible intervals is dense in $I_n$.
\end{enumerate}
\end{lemma}

\begin{proof}
(i). Recall from \cite{Kong_Li_2015} that $(q_{KL},M+1]\backslash\overline{\ub}=\bigcup (p_0,q_0)$, where each interval $(p_0,q_0)$ is of the form $(q_L(\sa),q_c(\sa))$ for some fundamental word $\sa$. Here $q_c(\sa):=\hat{\Phi}_\sa^{-1}(q_{KL}^*)\in J_\sa$ is the {\em de Vries-Komornik number} associated with $J_\sa$. Furthermore, $\overline{\ub}$ has Lebesgue measure zero. Hence the intervals $(p_0,q_0)$ are dense in $(q_{KL},M+1]$. But then the intervals $J_\sa=[q_L(\sa),q_R(\sa)]$ are certainly dense in $(q_{KL},M+1]$. Since the irreducible intervals are the maximal fundamental intervals in $(q_G,M+1]$, it follows that their union is dense in $(q_T,M+1]$.

(ii). By the same reasoning as above, the fundamental intervals inside $I_n$ are dense in $I_n$. Since the $n$-irreducible intervals in $I_n$ are the maximal fundamental intervals in $I_n$, the result follows.
\end{proof}

\begin{theorem} \label{thm:entropy-plateaus-detail}
\begin{enumerate}[(i)]
\item The plateaus of $H$ in $(q_T,M+1]$ are exactly the irreducible intervals that lie inside $(q_{KL},M+1]$.
\item The plateaus of $H$ in $I_n$ are exactly the $n$-irreducible intervals that lie inside $(q_{KL},M+1]$, for each $n\in\N$.
\end{enumerate}
\end{theorem}

\begin{proof}
Statement (i) is proved in \cite[Section 5.1]{AlcarazBarrera-Baker-Kong-2016}, but see also Section \ref{sec:irreducible-plateaus} below. We show how (ii) follows from (i) and the ideas in this article.
We proceed by induction. First, if we take $0$-irreducible to mean irreducible, statement (ii) holds for $n=0$ by part (i). Now let $k\in\N$, and suppose (ii) holds for all $n<k$. Let $J=[q_L,q_R]$ be a $k$-irreducible interval in $(q_{KL},M+1]$. Then $J\subset I_k$ by Proposition \ref{prop:location-of-n-irreducible-intervals}. Put $\hat{q}_L:=\hat{\Phi}_{\su}(q_L)$ and $\hat{q}_R:=\hat{\Phi}_{\su}(q_R)$. Then $\hat{J}:=[\hat{q}_L,\hat{q}_R]$ is a $(k-1)$-irreducible interval which lies inside $I_{k-1}$ by Lemma \ref{lem:Phi-hat-properties}, so it is an entropy plateau by the induction hypothesis. In particular, $H^*(\hat{q}_R)=H^*(\hat{q}_L)$. Then Proposition \ref{prop:entropy-bridge} implies that $H(q_R)=H(q_L)$, hence $H$ is constant on $J$. Next, take a point $r>q_R$. Then by Lemma \ref{lem:dense-intervals}(ii), there is a $k$-irreducible interval $J'=[q_L',q_R']$ with $q_R<q_L'<r$. Let $\hat{q}_L':=\hat{\Phi}_{\su}(q_L')$ and $\hat{q}_R':=\hat{\Phi}_{\su}(q_R')$. Then $\hat{J}':=[\hat{q}_L',\hat{q}_R']$ is an entropy plateau to the right of $\hat{J}$, so $H^*(\hat{q}_L')>H^*(\hat{q}_R)$. Hence, by Proposition \ref{prop:entropy-bridge} again, $H(q_L')>H(q_R)$. Thus $H(r)>H(q_R)$. In the same way, we can show that $H(p)<H(q_L)$ for every $p<q_L$. Therefore, $J$ is an entropy plateau.

Vice versa, every entropy plateau in $I_n$ must be an $n$-irreducible interval, because the $n$-irreducible intervals are dense in $I_n$.
\end{proof}

Proposition \ref{prop:entropy-bridge} gives new information about the entropy plateaus: Take $M=1$ for the moment. Then for any entropy plateau in $I_n$ (with $n\geq 1$) there is a corresponding entropy plateau in $I_{n-1}$ with twice the entropy. Vice versa, for any entropy plateau in $I_{n-1}$ there is a corresponding entropy plateau in $I_{n}$ with half the entropy. The graph of $H$ on $I_n$ looks much like a smaller version of the graph of $H$ on $I_{n-1}$, contracted vertically by a factor $2$ but badly distorted horizontally, as the map $\hat{\Phi}_{\su}$ is far from being linear.

By contrast, when $M\geq 2$ the graph of $H: q\mapsto h(\us_q)$ over $I_n$ for $n\geq 1$ looks more like a smaller copy of the graph of $H^*: q\mapsto h(\us_q^*)$ over $(q_T^*,2]$, rather than a smaller copy of the graph of $H$ itself over $(q_T,M+1]$. In fact, when $M$ is even, the values of $H$ on the entropy plateaus of $\us_q$ in $I_n$ are precisely the same as the values of $H^*$ on the entropy plateaus of $\us_q^*$ in $I_{n-1}^*$, for each $n\geq 1$. When $M$ is odd, they are half that big.

It also follows from Proposition \ref{prop:entropy-bridge} that the values of the functions $H_M: q\mapsto h(\us_q(M))$ for $q\in(q_T(M),M+1]$ and $M\geq 1$ completely determine the entropy of $\us_q$ for any $M$ and any $q\in(1,M+1]$.

\section{A shorter proof for irreducible plateaus} \label{sec:irreducible-plateaus}

In this section we give a more streamlined proof of Theorem \ref{thm:entropy-plateaus-detail}(i).  We first introduce the sets
\begin{equation*}
\vs_q:=\{(x_i)\in\Omega_M: \overline{\alpha(q)}\lle \sigma^n((x_i)) \lle \alpha(q)\textrm{ for all }n\ge 0\}.
\end{equation*}
Note that $\wus_q\subset \vs_q$. The point is that, while $\wus_q$ need not be a subshift of $\Omega_M$, $\vs_q$ always is. Moreover, $\vs_q\backslash\wus_q$ is countable, and $h(\wus_q)=h(\vs_q)$ (see \cite[Proposition 2.6]{Allaart-Kong-2018a}).

The key is to prove that for any irreducible interval $[q_L,q_R]$ in $[q_T,M+1]$, the subshift $\vs_{q_L}$ is transitive. Recall that a subshift $X$ of $\Omega_M$ is {\em (topologically) transitive} if, for any two words $u,v\in B_*(X)$, there is a word $w\in B_*(X)$ such that $uwv\in B_*(X)$. Here $B_*(X)$ denotes the set of all finite words occurring in sequences from $X$, including the empty word. We begin by proving the following crucial fact, which is Proposition 3.17 in \cite{AlcarazBarrera-Baker-Kong-2016}.

\begin{proposition} \label{prop:transitive}
Let $q\in(q_T,M+1]$. If $\alpha(q)$ is irreducible, then $\vs_q$ is a transitive subshift.
\end{proposition}

\begin{lemma} \label{lem:at-most-double}
Let $q>q_T$ be such that $\alpha(q)=a_1 a_2\dots$ is irreducible, and assume $a_1=M$. Then there is a strictly increasing sequence $(m_j)$ of positive integers such that for each $j$, $a_1\dots a_{m_j}^-$ is fundamental. Moreover, the sequence $(m_j)$ can be chosen so that $m_{j+1}\leq 2m_j$ for every $j$, and
\[
m_1=\begin{cases}
2, & \mbox{if $M=1$},\\
1, & \mbox{if $M\geq 2$}.
\end{cases}
\]
\end{lemma}

\begin{proof}
When $M=1$, $q>q_T$ implies that $a_1 a_2=11$, and $10$ is fundamental so we can take $m_1=2$. When $M\geq 2$, the assumption $a_1=M$ implies that $a_1^-$ is fundamental, so we can take $m_1=1$.

We now proceed by induction. Suppose $m_1,\dots, m_j$ have been constructed satisfying the required properties. In particular, $a_1\dots a_{m_j}^-$ is fundamental. Let $u=a_1\dots a_m$, where $m:=m_j$. Since $(a_i)$ is irreducible, $(a_i)\lge u\overline{u}^+\succ u\overline{u}$. Let $k:=\min\{i>m:a_i>\overline{a_{i-m}}\}$, so $m<k\leq 2m$. It is easy to see using Definition \ref{def:fundamental-words} and Lemma \ref{lem:quasi-greedy expansion-alpha-q} that $a_1\dots a_k^-$ is fundamental. Thus, we can set $m_{j+1}=k$.
\end{proof}

The next lemma is a restatement of \cite[Lemma 3.13]{AlcarazBarrera-Baker-Kong-2016}.

\begin{lemma} \label{lem:at-most-half}
Let $u=a_1\dots a_m$ be a word such that $u^-$ is fundamental. Suppose for some $i$ we have $a_{m-i+1}\dots a_m^-=\overline{a_1\dots a_i}$. Then $i\leq m/2$.
\end{lemma}

\begin{lemma} \label{lem:no-mirror-image}
Let $q>q_T$ and suppose $\alpha(q)=a_1 a_2\dots$ is irreducible, where $a_1=M$. Then there does not exist a word $\sb$ such that $\alpha(q)$ begins with $\sb\overline{\sb}$. Hence, for each $m\in\N$,
\[
a_1\dots a_{2m}\succ a_1\dots a_m \overline{a_1\dots a_m}.
\]
\end{lemma}

\begin{proof}
If $\alpha(q)$ begins with $\sb\overline{\sb}$, then in order for $\alpha(q)\in\vs$ we must have $\alpha(q)=\big(\sb\overline{\sb}\big)^\f$. This means $\sb\overline{\sb}$ is fundamental, which implies $\sb^-$ is fundamental. (Since $q>q_T$, the possible exception $M=\sb=1$ is ruled out.) But then $\sb\overline{\sb}=\sb^-\circ(10)$, contradicting that $\alpha(q)$ is irreducible in view of Proposition \ref{lem:irreducible-connection}.
\end{proof}

\begin{proof}[Proof of Proposition \ref{prop:transitive}]
It was shown in \cite{Komornik-Kong-Li-17} that $\vs_q$ is a subshift of $\Omega_M$ for every $q\in(1,M+1]$. Fix $q\in(q_T,M+1]$ such that $\alpha(q)$ is irreducible. 
Write $\alpha(q)=a_1 a_2\dots$. We may assume that $a_1=M$, as otherwise sequences from $\vs_q$ use only the digits $1,2,\dots,M-1$, and we are effectively in the case of a smaller alphabet.

Let $(m_j)$ be the sequence given by Lemma \ref{lem:at-most-double}. We will show by induction that

\bigskip
$(*)$ for each $j\in\N$ and for each word $v\in B_*(\vs_q)$, there is a word $w\in B_*(\vs_q)$ 

\ \ \ \ \ such that $a_1\dots a_{m_j}^-wv\in B_*(\vs_q)$. 

\bigskip
\noindent This suffices to prove transitivity, for the following reason: Suppose $u=u_1\dots u_k$ is an arbitrary word in $B_*(\vs_q)$. If
\[
\overline{a_1\dots a_{k-i}}\prec u_{i+1}\dots u_k \prec a_1\dots a_{k-i} \qquad\mbox{for all $0\leq i<k$},
\]
then $u$ and $v$ can be concatenated: $uv\in B_*(\vs_q)$. Otherwise, let $i_0$ be the smallest $i$ such that $u_{i+1}\dots u_k=a_1\dots a_{k-i}$ or $\overline{a_1\dots a_{k-i}}$. Then the above inequalities hold for all $i<i_0$. We may assume by symmetry that $u_{i_0+1}\dots u_k=a_1\dots a_{k-i_0}$. Then $u_{i_0+1}\dots u_k$ can be extended to a word of the form $a_1\dots a_{m_j}^-$ for some sufficiently large $j$. By $(*)$, this word can be connected to any $v\in B_*(\vs_q)$. Hence the original word $u$ can be connected to any $v\in B_*(\vs_q)$.

We now proceed to prove $(*)$.
Take first $j=1$. If $M=1$, then $m_1=2$ and $u:=a_1 a_2^-=10$. If $v$ begins with $1$, then $uv\in B_*(\vs_q)$; otherwise, $u1v\in B_*(\vs_q)$. If, on the other hand, $M\geq 2$, then $u:=a_1^-$ satisfies $\overline{a_1}<u<a_1$ (since $a_1=M$), and hence $uv\in B_*(\vs_q)$.
This finishes the basis for the induction. 

Next, let $k\geq 2$ and suppose $(*)$ has been proven for all $j<k$, so $a_1\dots a_{m_j}^-$ can be connected to any word $v$, for all $j<k$. Consider $u:=a_1\dots a_{m_k}^-$. If there is no $l<m_k$ such that
\[
a_{m_k-l+1}\dots a_{m_k}^-=\overline{a_1\dots a_l},
\]
then $u$ and $v$ can be connected directly: $uv\in B_*(\vs_q)$. Otherwise, let $l_0$ be the {\em largest} such $l$. Note by Lemma \ref{lem:at-most-half} that $l_0\leq m_k/2$, and hence by Lemma \ref{lem:at-most-double} that $l_0\leq m_{k-1}$. 

If $M=1$ and $l_0=1$, then we can connect $u$ and $v$ very easily. Namely, if $v$ begins with $1$, then $uv\in B_*(\vs_q)$; whereas otherwise $u1v\in B_*(\vs_q)$. So we can assume that either $M\geq 2$ or $l_0\geq 2$. We claim that $a_1\dots a_{l_0}^-$ is fundamental. To see this, note first that $a_{i+1}\dots a_{l_0}^-\prec a_{i+1}\dots a_{l_0}\lle a_1\dots a_{l_0-i}$ for $1\leq i<l_0$ by Lemma \ref{lem:irreducible-connection}. We must verify that
\begin{equation}
a_{i+1}\dots a_{l_0}^-\lge \overline{a_1\dots a_{l_0-i}}, \qquad 1\leq i<l_0.
\label{eq:fundamental-check-l0}
\end{equation}
Since $a_1\dots a_{m_k}^-$ is fundamental, we have that
\[
a_{i+1}\dots a_{l_0}a_{l_0+1}\dots a_{m_k}^-\lge \overline{a_1\dots a_{m_k-i}},
\]
whence $a_{i+1}\dots a_{l_0}\lge \overline{a_1\dots a_{l_0-i}}$. Suppose equality holds. Then
\[
a_1\dots a_{l_0-i}=\overline{a_{i+1}\dots a_{l_0}}=a_{m_k-l_0+i+1}\dots a_{m_k}^-,
\]
contradicting the admissiblity of $a_1\dots a_{m_k}^-$. Thus, we have \eqref{eq:fundamental-check-l0}.

It now follows that $(a_1\dots a_{l_0}^-)^\f\in\vs$. Since $(a_i)$ is irreducible, this implies
\[
(a_i)\succ a_1\dots a_{l_0}\big(\overline{a_1\dots a_{l_0}}^+\big)^\f.
\]
Hence, there is an integer $r\geq 0$ and a block $C$ of length $l_0$ with $C\succ \overline{a_1\dots a_{l_0}}^+$ such that $(a_i)$ begins with $a_1\dots a_{l_0}\big(\overline{a_1\dots a_{l_0}}^+\big)^r C$. 
Consider now the word
\[
\tilde{u}:=a_1\dots a_{m_k}^-\big(a_1\dots a_{l_0}^-\big)^{r+1},
\]
which is an extension of $u$. Note that $|\tilde{u}|=m_k+l_0(r+1)=:n$. It is clear from Definition \ref{def:fundamental-words} that $\tilde{u}_{i+1}\dots\tilde{u}_n\prec a_1\dots a_{n-i}$ for each $0\leq i<n$. We only need to show that 
\begin{equation}
\tilde{u}_{i+1}\dots\tilde{u}_n\succ \overline{a_1\dots a_{n-i}}
\label{eq:u-tilde-sandwich}
\end{equation}
for all $i<n-l_0$; the induction hypothesis will then imply that $\tilde{u}$ can be connected to any word $v$, since $l_0\leq m_{k-1}$ so $a_1\dots a_{l_0}^-$ can be connected to any word $v$. We break the verification of \eqref{eq:u-tilde-sandwich} in two cases:

\bigskip
{\em (a)} $i<m_k$. Suppose $a_{i+1}\dots a_{m_k}^-=\overline{a_1\dots a_{m_k-i}}$ (otherwise we are done). Then $m_k-i\leq l_0$ by the definition of $l_0$. If $m_k-i=l_0$, then by the choice of $r$,
\begin{align*}
\overline{\tilde{u}_{i+1}\dots\tilde{u}_n}&=a_1\dots a_{l_0}\big(\overline{a_1\dots a_{l_0}}^+\big)^{r+1}\\
&\prec a_1\dots a_{l_0}\big(\overline{a_1\dots a_{l_0}}^+\big)^r C\\
&=a_1\dots a_{n-i},
\end{align*}
so \eqref{eq:u-tilde-sandwich} holds. Assume therefore that $m_k-i<l_0$. Then
\[
\overline{\tilde{u}_{i+1}\dots\tilde{u}_{2m_k-i}}=a_1\dots a_{m_k-i}\,\overline{a_1\dots a_{m_k-i}}\prec a_1\dots a_{2m_k-i}
\]
by Lemma \ref{lem:no-mirror-image}, and again we obtain \eqref{eq:u-tilde-sandwich}.

\bigskip

{\em (b)} $m_k\leq i<n-l_0$. (Observe this case only happens if $r\geq 1$.) Write $i=m_k+tl_0+j$ where $t\in\{0,1,\dots,r-1\}$ and $0\leq j<l_0$. Then $\overline{\tilde{u}_{i+1}\dots\tilde{u}_n}$ begins with $\overline{a_{j+1}\dots a_{l_0}}^+$, and is followed by at least one block $\overline{a_1\dots a_{l_0}}^+$. Suppose $\overline{a_{j+1}\dots a_{l_0}}^+=a_1\dots a_{l_0-j}$ (otherwise we are done). We argue that $j$ cannot be zero, in other words, that $\overline{a_1\dots a_{l_0}}^+\prec a_1\dots a_{l_0}$. This is clear if $l_0\geq 2$ since $\overline{a_1}<a_1$; whereas if $l_0=1$ we have $M\geq 2$ by our earlier assumption, so $\overline{a_1}^+=\overline{M}^+=0^+=1<M=a_1$.

Thus, $j\geq 1$. But then $\overline{\tilde{u}_{i+1}\dots\tilde{u}_n}$ begins with $a_1\dots a_{l_0-j}\overline{a_1\dots a_{l_0-j}}$, which is smaller than $a_1\dots a_{2(l_0-j)}$ by Lemma \ref{lem:no-mirror-image}.

Since we verified \eqref{eq:u-tilde-sandwich} for all $i<n-l_0$, the proof is complete.
\end{proof}

\begin{remark}
The definition of $l_0$ in the above proof is closely related to the {\em reflection recurrence word} introduced in \cite{AlcarazBarrera-Baker-Kong-2016}. Alcaraz Barrera et al.~\cite{AlcarazBarrera-Baker-Kong-2016} use the reflection recurrence to directly construct the connecting word $w$ between $u$ and $v$. Their technique is very similar to ours. However, we prefer the above approach using induction, which allows us to keep technicalities to a minimum and bring the main ideas of the proof into better focus.
\end{remark}

From here, the proof goes essentially as in \cite[Section 5.1]{AlcarazBarrera-Baker-Kong-2016}. We present it here with a few more simplifications.

\begin{proposition} \label{prop:constant-entropy}
Let $[p_L,p_R]$ be an irreducible interval in $(q_{KL},M+1]$. Then $H(p_R)=H(p_L)$.
\end{proposition}

\begin{proof}
It is sufficient to show that $H(p_R)\leq H(p_L)$. Note that $h(\vs_q)=h(\us_q)$ for all $q$, so $H(q)=h(\vs_q)$. 

Let $[p_L,p_R]$ be generated by a fundamental word $\sa$ of length $m$, so $\alpha(p_L)=\sa^\f$. We argue first that 
\begin{equation}
H(p_L)\geq \frac{\log 2}{m}.
\label{eq:simple-lower-bound}
\end{equation}
Since $[p_L,p_R]$ is irreducible and $p_L>q_{KL}$, we have in fact that $p_L>q_T$, and so $H(p_L)\geq H(q_T)\geq (\log 2)/2$ by Corollary \ref{cor:entropy-of-qT}. This gives \eqref{eq:simple-lower-bound} when $m\geq 2$. On the other hand, if $m=1$, then we must have $M\geq 3$ and $\sa=a_1>\overline{a_1}$, since $p_L>q_T$. This implies $\{a_1,\overline{a_1}\}^\N\subset \vs_{p_L}$, so $H(p_L)\geq \log 2$ which gives \eqref{eq:simple-lower-bound} for $m=1$.

Let $h(\vs_{p_L})=\log\lambda$. Fix $\ep>0$. By the definition of $h(\vs_{p_L})$, there is a constant $C_1$ such that
\begin{equation}
\#B_n(\vs_{p_L})\leq C_1(\lambda+\ep)^n \qquad\forall\,n\geq 0.
\label{eq:word-estimate1}
\end{equation}
(Alcaraz Barrera et al. use the Perron-Frobenius theorem here, but this is not necessary.) On the other hand, if $(x_i)\in \vs_{p_R}\backslash \vs_{p_L}$, then there is an index $j$ such that $x_{j+1}\dots x_{j+m}=a_1\dots a_m$ or $\overline{a_1\dots a_m}$. This implies that $\sigma^j((x_i))\in X_\sa$ (see \cite[Lemma 3.2]{Allaart-Kong-2018b}). From the definition of $X_\sa$ it follows that there is constant $C_2$ such that
\begin{equation}
\#B_n(X_\sa)\leq C_2 2^{n/m} \qquad\forall\,n\geq 0.
\label{eq:word-estimate2}
\end{equation}
(In fact, one may take $C_2=2$.) By \eqref{eq:simple-lower-bound}, $\lambda\geq 2^{1/m}$. We thus obtain, using \eqref{eq:word-estimate1} and \eqref{eq:word-estimate2},
\begin{align*}
\#B_n(\vs_{p_R}) &\leq \sum_{j=0}^n \#B_j(\vs_{p_L})\#B_{n-j}(X_\sa)
\leq C_1 C_2\sum_{j=0}^n (\lambda+\ep)^j 2^{(n-j)/m}\\
&\leq C_1 C_2\sum_{j=0}^n (\lambda+\ep)^n=C_1 C_2(n+1)(\lambda+\ep)^n.
\end{align*}
Hence, $H(p_R)=h(\vs_{p_R})\leq \log(\lambda+\ep)$. Letting $\ep\to 0$ completes the proof.
\end{proof}

We now complete the proof of Theorem \ref{thm:entropy-plateaus-detail}(i). First, let $[p_L,p_R]$ be an irreducible interval in $(q_{KL},M+1]$. By Proposition \ref{prop:constant-entropy}, $H$ is constant on $[p_L,p_R]$. By Lemma \ref{lem:dense-intervals}, the irreducible intervals are dense in $(q_T,M+1]$. Since irreducible intervals are mutually disjoint, this means that for every $\ep>0$ there is an irreducible interval $[q_L,q_R]$ such that $p_R<q_L<p_R+\ep$. Recall that $\vs_{q_L}$ is a subshift of finite type, and by Proposition \ref{prop:transitive}, it is transitive. Furthermore, $\vs_{p_L}$ is a proper subshift of $\vs_{q_L}$. Thus, by \cite[Corollary 4.4.9]{Lind_Marcus_1995}, $H(p_L)<H(q_L)$. It follows that $H(p_R+\ep)\geq H(q_L)>H(p_R)$. By similar reasoning, $H(p_L-\ep)<H(p_L)$ for every $\ep>0$. Hence, $[p_L,p_R]$ is an entropy plateau.

Vice versa, every entropy plateau in $(q_T,M+1]$ must be an irreducible interval, because the irreducible intervals are dense in $(q_T,M+1]$. This completes the proof.

\section*{Acknowledgment}
The author wishes to thank Derong Kong for helpful comments on an earlier draft of the manuscript, and two anonymous referees for their careful reading of the paper and for several helpful suggestions that led to an improved presentation.

\end{document}